\documentclass[11pt]{amsart}
\usepackage{fullpage}
\usepackage{amsmath,amssymb,amsthm,graphicx,url}
\usepackage{booktabs}
\usepackage{algorithm,algpseudocode}

\theoremstyle{definition}
\newtheorem{theorem}{Theorem}
\newtheorem{lemma}[theorem]{Lemma}
\newtheorem{cor}[theorem]{Corollary}

\newtheorem{example}[theorem]{Example}

\newcommand{\eb}{\overline{b}} 
\newcommand{\el}{\overline{l}} 
\newcommand{\db}{\mathring{b}} 
\newcommand{\dl}{\mathring{l}} 

\makeatletter
\def\imod#1{\allowbreak\mkern10mu({\operator@font mod}\,\,#1)}
\makeatother

\begin{document}

\title{Matching Expectations}

\date{\today}

\author{Daniel J. Velleman}
\address{Dept. of Mathematics\\
  Amherst College\\
  Amherst, MA 01002}
\email{djvelleman@amherst.edu}

\author{Gregory S. Warrington}
\address{Dept. of Mathematics and Statistics\\
  University of Vermont \\
  Burlington, VT 05401}
\email{gregory.warrington@uvm.edu}

\thanks{Second author partially supported by a grant from the Simons
  Foundation (197419 to GSW)}

\vspace*{.3in}

\begin{abstract}
  The game of memory is played with a deck of $n$ pairs of cards.  The
  cards in each pair are identical.  The deck is shuffled and the
  cards laid face down.  A move consists of flipping over first one
  card then another.  The cards are removed from play if they match.
  Otherwise, they are flipped back over and the next move commences.
  A game ends when all pairs have been matched.  We determine that,
  when the game is played optimally, as $n\rightarrow \infty$:
  \begin{itemize}
  \item The expected number of moves is $(3 - 2\ln 2)n +
    7/8 - 2 \ln 2 \approx 1.61 n$.
  \item The expected number of times two matching cards are unwittingly
    flipped over is $\ln 2$.
  \item The expected number of flips until two matching cards have
    been seen is $2^{2n}/\binom{2n}{n} \sim \sqrt{\pi n}$.
  \end{itemize}
\end{abstract}

\maketitle

\section{Introduction}

Any respectable list of desert island card games must include war to
pass the time (the expected length of a game is infinite, after
all~\cite{war}).  But to maintain your remembrance of things past,
it's hard to beat the game of memory.  Memory is played with a deck of
$2n$ cards.  The cards are numbered from $1$ up to $n$, with each
number appearing twice.  The deck is shuffled and then the cards are
laid face down in a tableau. A \emph{move} consists of flipping first
one card and then a second.  If the cards match, both are removed from
play and the current player moves again.  Otherwise, they are flipped
back over and play passes to the next player.  Play ends when all
pairs have been removed; the player who has removed the most pairs is
the winner.  The \emph{length} of a game is the total number of moves.

Perhaps surprisingly, strategy plays a role in the two-player version
of the game~\cite{strategy}.  In some situations it may not be to a
player's advantage to acquire new information by flipping over a card
that has not been flipped before.  The reason is that her opponent
also acquires the same information.

In this exploration we consider the solitaire version of memory, in
which strategy plays no part.  Many games, such as poker and go fish,
do not fair well as one-person games.  Memory does, as long as one
shifts the goal from collecting the largest number of pairs to
collecting all pairs in the fewest moves.  As an additional
simplification, we assume that the player has perfect memory.

You may find it helpful to play a few games of solitaire memory on
your own.  You can use some subset of an ordinary deck of cards; play
online at a number of websites, such as \cite{classicmem} or
\cite{mathfun}; or, as of this writing, find a computer version on the
wall at the airport (at least in Houston or Minneapolis).

In the solitaire version of memory, acquiring information by flipping
over an unknown card is always helpful.  It is therefore not hard to
show that an optimal strategy is to proceed as described in
Algorithm~\ref{alg:strategy}.  Throughout this paper, we assume that
the player is using this optimal strategy.  We consider three
questions regarding this one-person game of memory.\\

\begin{algorithm}
\caption{Optimal strategy}\label{alg:strategy}
\begin{algorithmic}[1] 
  \If{the positions of both cards of a pair are known}
    \State flip them over and remove them
  \Else
    \State flip over an unknown card
    \If{the location of its mate is known}
    \State flip the mate and remove both
    \Else
      \State flip over another unknown card
      \State (and remove the pair if you are lucky enough to find the mate)
    \EndIf
  \EndIf
\end{algorithmic}
\end{algorithm}



\noindent
\textbf{1. What is the expected length of a game (i.e., how many moves
  are required)? }\\

Each card must be flipped over at least once.  If it is not matched
the first time it is flipped over, then it is flipped again, and
removed, once the location of its mate is known.  Thus, each card is
flipped either once or twice, and therefore the total number of card
flips is between $2n$ and $4n$.  Since two cards are flipped over on
each move, this means that the number of moves in any game will be
between $n$ and $2n$.  Analysis of the possibilities for the end of
the game shows that the last card to be flipped over will only be
flipped once, and therefore the length of the game cannot be $2n$.
Therefore, the length lies between $n$ and $2n-1$.
Example~\ref{ex:extremal} (in Section~\ref{sec:setup}) illustrates
that both of these lengths are possible.

\begin{figure}[htbp]
  \centering 
  \includegraphics{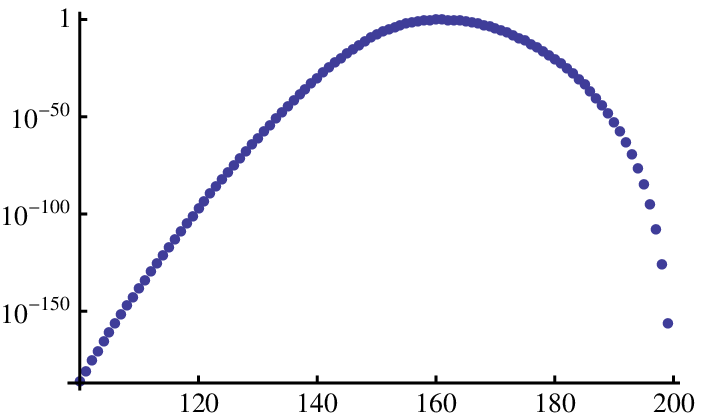} \qquad \includegraphics{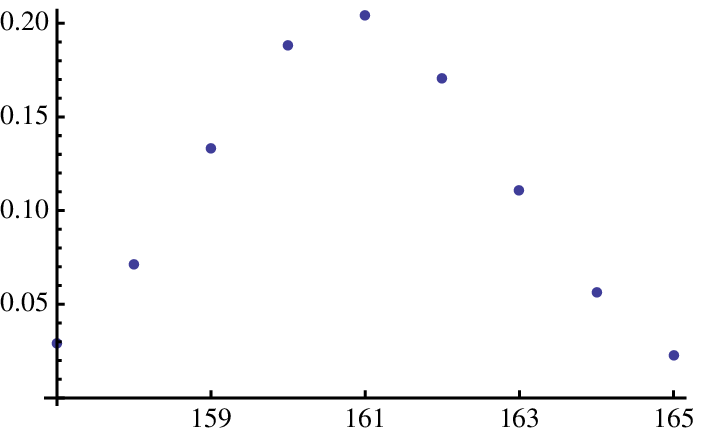}
  \caption{On the left is the distribution of game lengths for
    $n=100$, with logarithmic $y$-scale.  On the right we use a linear
    $y$-scale and a restricted domain to highlight the fact that over
    $98\%$ of games for $n=100$ will have a length lying between $157$
    and $165$.}
  \label{fig:dist-percentage}
\end{figure}

Computer investigations indicate that the expected length of a game is
roughly $1.6 n$.  In Sections~\ref{sec:explen} and \ref{sec:limsums}
we will determine the exact length, and show that it is approximately
$(3 - 2\ln 2)n + 7/8 - 2 \ln 2$, with the error in this approximation
approaching 0 as $n\to \infty$.  Figure~\ref{fig:dist-percentage} shows the distribution of game lengths for $n=100$.  The expected length in this case is about 160.8589,
and our approximation gives about 160.8593.

A \emph{lucky} move is one in which the player flips two cards that
have not been flipped before and they happen to match.  Note that a
game of length $n$ is one in which every move is lucky.  The second
question, to be addressed in Section~\ref{sec:sumdleven}, is:\\

\noindent
\textbf{2. What is the expected number of lucky moves in a game?}\\

Early in the game, the player has little information about the
locations of the various cards.  In all likelihood, each card flipped
over will be the first of its pair to be seen.  But certainly by the
$(n+1)$th flip, the player will begin encountering the mates of cards
she has already seen.  This brings us to our third and final question,
to be
addressed in Section~\ref{sec:repeat}:\\

\noindent
\textbf{3. How many flips are required before the player should expect
  to have seen both cards of a pair?}

\section{Combinatorial setup}
\label{sec:setup}

A game of memory is not affected by how the cards are physically laid
out.  In light of this, it will be convenient to assume that the cards
are arranged in a single row.  And since the cards are shuffled before
they are laid out, when our strategy calls for the player to flip an
unknown card, there is no advantage to flipping one particular unknown
card in preference to another.  We may therefore assume that unknown
cards are simply flipped from left to right.  With this convention, game
play depends only on the order in which the cards are laid out.

\begin{example}\label{ex:n=6ex}
Suppose that $n=6$ and the cards are laid out in the order
\[
1\ 2\ 1\ 6\ 2\ 3\ 3\ 5\ 4\ 4\ 5\ 6.
\]
Play will proceed as follows:
\begin{enumerate}
\item The player flips over the first two cards, which are a 1 and a
  2.  Since they don't match, they are flipped back over.
\item The player flips over the next card, which is another 1.
  Remembering that the first card was also a 1, the player flips the
  first card to get the match and removes the two 1s.
\item The player flips the next two cards, which are a 6 and a 2.
  They don't match, so they are flipped back.
\item The player now knows the locations of both 2s.  She therefore
  flips them and removes them.
\item The player flips the next two cards, which are both 3s.  They
  match, so she removes them.  Since these cards had not been flipped
  before, this is an example of a lucky move.
\item The player flips the next two cards: a 5 and a 4.  They don't
  match, so they are flipped back.
\item The player flips the next card, which is the second 4.
  Remembering that the previous card was also a 4, she flips it and
  removes the two 4s.
\item The player flips the second 5, flips the first 5 again, and
  removes the two 5s.
\item The player flips the second 6, flips the first 6 again, and
  removes the two 6s.
\end{enumerate}
\end{example}

\begin{example}\label{ex:extremal}
  We can now give examples of the shortest and longest possible games.
  If the cards are laid out in the order
  \[
  1\ 1\ 2\ 2\ 3\ 3\ \cdots\ n\ n,
  \]
  then the game will consist of $n$ lucky moves; this is the shortest
  possible game.  It is only slightly harder to illustrate the longest
  possible game.  We will let you verify that if the cards are laid
  out in the order
  \[
  1\ 2\ 3\ 1\ 4\ 2\ \cdots\ k\ \ k-2\ \cdots\ n-1\ \ n-3\ \ n\ \ n-2\ \ n-1\ \ n,
  \]
then it will take $2n-1$ moves to complete the game.
\end{example}

\begin{figure}[htbp]
  \centering {\scalebox{1}{\includegraphics{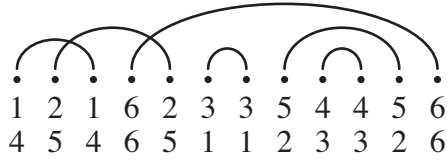}}}
  \caption{Two deals for $n=6$ with the same interconnection network.}
  \label{fig:icnetwork}
\end{figure}

Although we have presented all of our examples by specifying the order
in which the cards are laid out, in fact game play depends on only the
relative positions of the pairs.  These relative positions can be
visualized using the notion of an \emph{interconnection network}.  An
interconnection network~\cite{fs, loz} consists of $2n$ points along
with $n$ endpoint-disjoint arcs that connect pairs of points.
(Equivalently, such a network encodes an involution without fixed
points.)  Figure~\ref{fig:icnetwork} illustrates two deals for $n=6$
along with the associated interconnection network, in which points
labeled with the same number are connected by an arc.  (The first of
these deals is the one in Example~\ref{ex:n=6ex} above.)

The two deals in Figure \ref{fig:icnetwork} differ only in the labels
we have assigned to each pair of cards, and therefore the first deal
leads to essentially the same game as the second.  We wish to identify
those games with the same underlying interconnection network.  In
general, there is an $n!$-to-$1$ map from size-$2n$ deals to size-$2n$
interconnection networks.  This map for $n=2$ is illustrated in
Figure~\ref{fig:icn-n2}.  We prefer to remain in the realm of card
games, so we designate a representative for each class of deals.  We
call a deal \emph{standard} if, when the game is played, the pairs are
removed in order from 1 to $n$.  Put another way, a deal is standard
if the second occurrence of $i$ occurs to the left of the second
occurrence of $i+1$ for each $i$ between 1 and $n$.  The deals in
Examples~\ref{ex:n=6ex} and~\ref{ex:extremal} are all standard.

\begin{figure}[htbp]
  \centering {\scalebox{1}{\includegraphics{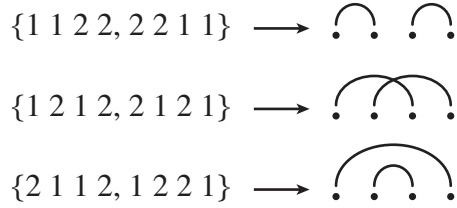}}}
  \caption{Memory deals for $n=2$ paired according to interconnection network.}
  \label{fig:icn-n2}
\end{figure}

Let $M_n$ denote the set of standard permutations of the multiset
$\{1,1,2,2,\ldots,n,n\}$.  At this point, we have reduced a game of
memory to an element of some $M_n$.  We write these elements in
one-line notation (reminiscent of an actual row of cards).  For
example, define $\sigma \in M_2$ by $\sigma(1)=\sigma(4)=2$ and
$\sigma(2)=\sigma(3)=1$.  The element $\sigma$ represents the game
$2\,1\,1\,2$, in which the first and fourth cards are 2s and the
second and third are 1s.

For completeness, we include a proof of the following formula for
the number of inequivalent games of memory---that is, the cardinality of $M_n$.

\begin{lemma}
  For $n\geq 1$, the number of standard games of memory on a deck of
  $2n$ cards is
  \begin{equation*}
    (2n-1)!! = 1\cdot 3\cdot 5 \cdots (2n-1) = \frac{(2n)!}{2^n n!}.
  \end{equation*}
\end{lemma}
\begin{proof}
  The first equality involving the \emph{double factorial} function is
  true by definition.  The second equality can be checked by simple
  algebra.  We prove by induction that the number of games is given by
  the above product of odd integers.

  For $n=1$ the formula is trivially true.  Now suppose that $n \ge
  2$.  Any standard deal of a deck of size $2n$ can be constructed by
  adding two copies of $n$ to a standard deal of a deck of size
  $2(n-1)$.  In order for the augmented deal to be standard, the
  second occurrence of $n$ must be at the end of the row.  Thus, there
  are $2(n-1)+1 = 2n-1$ possible locations for the first occurrence of
  $n$.  The lemma therefore follows by the principle of mathematical
  induction.
\end{proof}

\section{Expected position of the first match} %
\label{sec:repeat}

Fix $n \ge 1$.  For any $\sigma \in M_n$, let $f(\sigma)$ denote the
smallest index $j$ such that $\sigma(j) = \sigma(i)$ for some $i < j$.
In other words, $f(\sigma)$ is the position of the first card in
$\sigma$ that matches a previous card.  Since we are assuming that the
deal is standard, this first matching card will be a 1.  Let $a(n,j)$
count the number of permutations $\sigma\in M_n$ for which $f(\sigma)
= j$.  Notice that $f(\sigma)$ is always between 2 and $n+1$, so
$a(n,j)=0$ if $j<2$ or $j > n+1$.  Then for a randomly chosen $\sigma
\in M_n$, the probability that $f(\sigma) =j$ is
\[
\frac{a(n,j)}{(2n)!/(2^nn!)},
\]
and therefore the expected value of $f(\sigma)$ is
\begin{equation}\label{eq:expfsigma}
  \sum_{j=2}^{n+1} j \cdot \frac{a(n,j)}{(2n!)/(2^nn!)} = 
  \frac{\sum_{j=2}^{n+1} j \cdot a(n,j)}{(2n)!/(2^nn!)}.
\end{equation}

It turns out that there is a simple formula for the numerator of the
last fraction:
\begin{lemma}\label{lem:expnum}
For every positive integer $n$,
\[
\sum_{j=2}^{n+1} j \cdot a(n,j) = 2^n n!.
\]
\end{lemma}
\begin{proof}
  We proceed by induction on $n$.  The equation is easily seen to be
  correct when $n=1$.  Now suppose that $n\ge 2$, and the equation
  holds for $n-1$.

  A permutation $\sigma \in M_n$ with $f(\sigma) = j$ can be
  constructed from a permutation in $M_{n-1}$ in one of two ways.  The
  first way is to start with a permutation $\sigma' \in M_{n-1}$ with
  $f(\sigma')=j$, insert an $n$ into $\sigma'$ somewhere after
  position $j$, and then add a second $n$ at the end.  The number of
  permutations that can be constructed in this way is $(2n-1-j) \cdot
  a(n-1,j)$.  The second way is to start with a permutation $\sigma'
  \in M_{n-1}$ with $f(\sigma') = j-1$, insert an $n$ before position
  $j-1$, and then add another $n$ at the end.  This can be done in
  $(j-1) \cdot a(n-1,j-1)$ ways.  Thus, we have the recurrence
\[
a(n,j) = (2n-1-j) \cdot a(n-1,j) + (j-1) \cdot a(n-1,j-1).
\]

We can now compute:
\begin{align*}
\sum_{j=2}^{n+1} j \cdot a(n,j) &= \sum_{j=2}^{n+1} j \cdot [(2n-1-j) \cdot a(n-1,j) + (j-1) \cdot a(n-1,j-1)]\\
&= \sum_{j=2}^{n+1} j \cdot (2n-1-j) \cdot a(n-1,j) + \sum_{j=2}^{n+1} j \cdot (j-1) \cdot a(n-1,j-1)\\
&= \sum_{j=2}^{n+1} j \cdot (2n-1-j) \cdot a(n-1,j) + \sum_{j=1}^n (j+1) \cdot j \cdot a(n-1,j)\\
&= \sum_{j=2}^n 2nj \cdot a(n-1,j) = 2n \cdot \sum_{j=2}^n j \cdot a(n-1,j)\\
&= 2n \cdot 2^{n-1} (n-1)! \qquad\qquad \text{(by the inductive hypothesis)}\\
&= 2^n n!. \qedhere
\end{align*}
\end{proof}

\begin{theorem}
  For $n\geq 1$, the expected position of the first match is
\[
\frac{2^{2n}}{\binom{2n}{n}}.
\]
\end{theorem}
\begin{proof}
  Using equation~\eqref{eq:expfsigma} and Lemma~\ref{lem:expnum}, we
  find that the expected position of the first match is
\[
\frac{\sum_{j=2}^{n+1} j \cdot a(n,j)}{(2n)!/(2^n n!)} = \frac{2^n n!}{(2n)!/(2^n n!)} = \frac{2^{2n}}{(2n)!/n!^2} = \frac{2^{2n}}{\binom{2n}{n}}. \qedhere
\]
\end{proof}

\begin{cor}
  The expected position of the first match grows as $\sqrt{\pi n}$.
\end{cor}
\begin{proof}
  Straightforward applications of Stirling's approximation $n!\sim
  \sqrt{2\pi n}\left(\frac{n}{e}\right)^n$ (see, for example,
  ~\cite{concmath}) yield
  \begin{equation*}
    \frac{2^{2n}}{\binom{2n}{n}} \sim
    2^{2n}\frac{(\sqrt{2\pi n}\left(\frac{n}{e}\right)^n)^2}
    {\sqrt{2\pi (2n)}\left(\frac{2n}{e}\right)^{2n}}
     = \sqrt{\pi n}. \qedhere
  \end{equation*}
\end{proof}

\section{Expected length of a game}
\label{sec:explen}

Our main goal in this section is to prove the following theorem:

\begin{theorem}\label{thm:main}
  The expected length of a game with $2n$ cards is $(3 - 2\ln 2)n +
  7/8 - 2 \ln 2 + \epsilon_n$, where $\lim_{n \to \infty} \epsilon_n =
  0$.
\end{theorem}

Consider any standard deal $\sigma \in M_n$.  In the game based on
$\sigma$, the player will turn over unknown cards two at a time until
she comes to the first matching card at position $f(\sigma)$.  What
happens at that point depends on whether $f(\sigma)$ is even or odd.
If $f(\sigma)$ is odd, then the card at position $f(\sigma)$, which is
the second 1, will be flipped over at the beginning of a move.  Since
we are assuming that the player has perfect memory, she will then flip
the first 1 and remove both 1s.  The number of moves required to flip
all the cards up to position $f(\sigma)$ and then remove the matching
1s is therefore $(f(\sigma)+1)/2 = f(\sigma)/2 + 1/2$.  On the other
hand, if $f(\sigma)$ is even then the card at position $f(\sigma)$
will be the second card flipped in a move.  If the first card flipped
in that move was also a 1 then the move is lucky and the two 1s are
removed right away.  If not, then it will take an additional move to
flip the two 1s and remove them.  Thus, if $f(\sigma)$ is even then
the number of moves to flip the first $f(\sigma)$ cards and remove the
two 1s is either $f(\sigma)/2$, if the first 1 is at position
$f(\sigma)-1$, or $f(\sigma)/2+1$, if not.

Once the two 1s have been removed, the player continues by flipping
the unknown card at position $f(\sigma)+1$, and she again flips
unknown cards two at a time until she reaches the next matching card,
the second 2.  If the second 2 is at position $f(\sigma)+k$, then as
above the number of moves required to flip all of the cards from
position $f(\sigma)+1$ to $f(\sigma)+k$ and remove the matching 2s is
$k/2+1/2$ if $k$ is odd, $k/2$ if $k$ is even and the first 2 is at
position $f(\sigma)+k-1$, and $k/2+1$ otherwise.  Once the 2s have
been removed, the player flips unknown cards two at a time until she
finds the second 3, and she continues in this way until all the pairs
have been matched and removed.

Thus, we see that the second occurrences of each of the numbers from 1
to $n$ can be thought of as breaking the deal $\sigma$ into $n$
blocks, with each block ending with a second occurrence.  The function
$f(\sigma)$ that we studied in the last section is the length of the
first block, but we will have to consider the lengths of all of the
blocks to determine the number of moves in the game.  Let the lengths
of the blocks be $b_1(\sigma) = f(\sigma)$, $b_2(\sigma)$, \ldots,
$b_n(\sigma)$, and notice that since the entire deck contains $2n$
cards, $b_1(\sigma) + b_2(\sigma) + \cdots + b_n(\sigma) = 2n$.

We will call a block \emph{lucky} if the block has length at least 2
and the last two cards in the block match.  The analysis above shows
that the number of moves required to flip all of the cards in the
$i$th block and then remove the two cards labeled $i$ is
$b_i(\sigma)/2 + a_i(\sigma)$, where $a_i(\sigma)$ is 1/2 if
$b_i(\sigma)$ is odd, 0 if $b_i(\sigma)$ is even and the $i$th block
is lucky, and 1 if $b_i(\sigma)$ is even and the $i$th block is not
lucky.  Thus, if $e(\sigma)$ is the number of blocks in $\sigma$ of
even length and $l(\sigma)$ is the number of these even-length blocks
that are lucky, then the length of the game is
\[
\sum_{i=1}^n \left(\frac{b_i(\sigma)}{2} + a_i(\sigma)\right) = \frac{\sum_{i=1}^n b_i(\sigma)}{2} + (n-e(\sigma)) \cdot \frac{1}{2} + (e(\sigma)-l(\sigma)) \cdot 1 = \frac{3n}{2} + \frac{e(\sigma)}{2} - l(\sigma).
\]

For example, if $n=6$ and $\sigma$ is the deal given in
Example~\ref{ex:n=6ex}, then the lengths of the blocks are 3, 2, 2, 3,
1, 1, and the third and fourth blocks are lucky.  Therefore $e(\sigma)
= 2$ and $l(\sigma)=1$, and the length of the game is
\[
\frac{3 \cdot 6}{2} + \frac{2}{2} - 1 = 9,
\]
which is in agreement with our analysis of this game in Example~\ref{ex:n=6ex}.

To prove Theorem~\ref{thm:main} we will need to find the expected
values of $e(\sigma)$ and $l(\sigma)$: the expected number of even
blocks and the expected number of lucky even blocks.  To compute
these, we will keep track of the number of blocks of each kind and
each length among all of the $(2n)!/(2^nn!)$ standard deals.  For any
positive integer $j$, let $b(n,j)$ be the total number of blocks of
length $j$ in all standard deals of a deck of $2n$ cards, and for $j
\ge 2$ let $l(n,j)$ be the total number of lucky blocks of length $j$.
The greatest possible length for a block is $n+1$, so $b(n,j) = l(n,j)
= 0$ for $j \ge n+2$.  The expected number of blocks of length $j$ is
then
\[
\eb(n,j) = \frac{b(n,j)}{(2n)!/(2^n n!)},
\]
and similarly the expected number of lucky blocks of length $j$ is
\[
\el(n,j) =
\frac{l(n,j)}{(2n)!/(2^n n!)}.
\]
Using this notation, we can write the expected value of $e(\sigma)$ as
$\sum_{i=1}^\infty \eb(n,2i)$, and the expected value of $l(\sigma)$
is $\sum_{i=1}^\infty \el(n,2i)$.  Notice that, although we have
written these expected values as infinite sums, in fact the sums are
finite, since all but finitely many terms in each sum are 0.  Thus the
expected length of a game is
\begin{equation}\label{eq:lengthform}
\frac{3n}{2} + \frac{\sum_{i=1}^\infty \eb(n,2i)}{2} - \sum_{i=1}^\infty \el(n,2i).
\end{equation}

There are several steps required to pass from
equation~(\ref{eq:lengthform}) to Theorem~\ref{thm:main}.  The first
is to replace the $\eb(n,j)$ and $\el(n,j)$ in~(\ref{eq:lengthform})
with (what turn out to be) their asymptotic limits.  We will show
that, for large $n$,
\begin{equation}\label{eq:blasymp}
\eb(n,j) \approx \frac{2(2n+1)}{j(j+1)(j+2)}\qquad\text{ and for $j \ge 2$, } \qquad
\el(n,j) \approx \frac{1}{j(j-1)}.
\end{equation}
We therefore define
\begin{equation}\label{eq:dbldef}
\db(n,j) = \eb(n,j) - \frac{2(2n+1)}{j(j+1)(j+2)}\qquad\text{ and for $j \ge 2$, }\qquad
\dl(n,j) = \el(n,j) - \frac{1}{j(j-1)}.
\end{equation}
Equivalently,
\begin{equation*}
\eb(n,j) = \frac{2(2n+1)}{j(j+1)(j+2)} + \db(n,j) \qquad\text{ and for $j \ge 2$, }\qquad
\el(n,j) = \frac{1}{j(j-1)} + \dl(n,j).
\end{equation*}

We wish to rewrite \eqref{eq:lengthform} using the asymptotic limits
of the $\eb(n,j)$ and $\el(n,j)$. To this end, we first compute
\[
\sum_{i=1}^\infty \eb(n,2i) = \sum_{i=1}^\infty \left(\frac{2(2n+1)}{2i(2i+1)(2i+2)} + \db(n,2i)\right) = \sum_{i=1}^\infty \frac{2(2n+1)}{2i(2i+1)(2i+2)} + \sum_{i=1}^\infty \db(n,2i).
\]
Although the sum on the left-hand side is finite, the two sums on the
right-hand side are infinite.  However, it is not hard to see that
these two infinite sums converge.  Indeed, we can evaluate the first
one by using the partial fractions decomposition
\begin{equation}\label{eq:bpartfrac}
\frac{2}{j(j+1)(j+2)} = \frac{1}{j} - \frac{2}{j+1} + \frac{1}{j+2}.
\end{equation}
This gives us
\begin{align*}
&\sum_{i=1}^\infty \frac{2(2n+1)}{2i(2i+1)(2i+2)} = (2n+1) \lim_{N \to \infty} \sum_{i=1}^N \left(\frac{1}{2i} - \frac{2}{2i+1} + \frac{1}{2i+2}\right)\\
&\qquad\qquad\qquad = (2n+1) \lim_{N \to \infty} \left(\frac{1}{2} - \frac{2}{3} + \frac{1}{4} + \frac{1}{4} - \frac{2}{5} + \frac{1}{6} + \cdots + \frac{1}{2N} - \frac{2}{2N+1} + \frac{1}{2N+2}\right)\\
&\qquad\qquad\qquad = (2n+1) \lim_{N \to \infty} \left(\frac{1}{2} - \frac{2}{3} + \frac{2}{4} - \frac{2}{5} + \cdots + \frac{2}{2N} - \frac{2}{2N+1} + \frac{1}{2N+2}\right)\\
&\qquad\qquad\qquad = (2n+1) \lim_{N \to \infty} \left(\frac{3}{2} + \frac{1}{2N+2} - 2\left(1 - \frac{1}{2} + \frac{1}{3} - \frac{1}{4} + \cdots + \frac{1}{2N+1}\right)\right)\\
&\qquad\qquad\qquad = (2n+1)\left(\frac{3}{2} - 2\ln 2\right),
\end{align*}
since the alternating harmonic series converges to $\ln 2$.  Therefore
\begin{equation}\label{eq:sumebeven}
\sum_{i=1}^\infty \eb(n,2i) = (2n+1)\left(\frac{3}{2}-2\ln 2\right) + \sum_{i=1}^\infty \db(n,2i).
\end{equation}

Next, we compute
\[
\sum_{i=1}^\infty \el(n,2i) = \sum_{i=1}^\infty \left(\frac{1}{2i(2i-1)} + \dl(n,2i)\right) =  \sum_{i=1}^\infty \frac{1}{2i(2i-1)} + \sum_{i=1}^\infty \dl(n,2i).
\]
For the first part we again use a partial fractions decomposition:
\[
\frac{1}{j(j-1)} = \frac{1}{j-1} - \frac{1}{j}.
\]
This gives us
\begin{align*}
\sum_{i=1}^\infty \frac{1}{2i(2i-1)} &= \lim_{N \to \infty} \sum_{i=1}^N \left(\frac{1}{2i-1} - \frac{1}{2i}\right)\\
&= \lim_{N \to \infty} \left(1 - \frac{1}{2} + \frac{1}{3} - \frac{1}{4} + \cdots + \frac{1}{2N-1} - \frac{1}{2N}\right) = \ln 2,
\end{align*}
so
\begin{equation}\label{eq:sumeleven}
\sum_{i=1}^\infty \el(n,2i) = \ln 2 + \sum_{i=1}^\infty \dl(n,2i).
\end{equation}

Combining \eqref{eq:lengthform}, \eqref{eq:sumebeven}, and
\eqref{eq:sumeleven}, we see that the expected length of the game is
\begin{align*}
\frac{3n}{2} + \frac{\sum_{i=1}^\infty \eb(n,2i)}{2} - \sum_{i=1}^\infty \el(n,2i)
&= (3-2\ln 2)n + \frac{3}{4} - 2\ln 2 + \frac{\sum_{i=1}^\infty \db(n,2i)}{2} - \sum_{i=1}^\infty \dl(n,2i).
\end{align*}

In the next section, we will show that
\[
\lim_{n \to \infty} \sum_{i=1}^\infty \db(n,2i) = \frac{1}{4} \qquad \text{ and } \qquad \lim_{n \to \infty} \sum_{i=1}^\infty \dl(n,2i) = 0.
\]
Substituting these values in will complete the proof of Theorem~\ref{thm:main}.

\section{Evaluation of Limits}
\label{sec:limsums}

In order to determine the limits required to complete the proof of Theorem~\ref{thm:main}, we need to establish some facts about the $\db(n,j)$
and $\dl(n,j)$.  To establish these facts, we begin by finding effective recurrences that can be used to compute $\db(n,j)$ and $\dl(n,j)$.

\subsection{Sum of $\db(n,2i)$}

First we determine recurrences for the $b(n,j)$.  If $n=1$, then the only standard deal is $1\ 1$, which has a single block of length 2.  So $b(1,2)=1$ and $b(1,j)=0$ for $j \ne 2$.
For $n \ge 2$, as we saw before, any standard deal $\sigma \in M_n$ can be constructed from a standard deal $\sigma' \in M_{n-1}$ by inserting an $n$ in one of the $2n-1$ possible positions in $\sigma'$ and then adding a second $n$ at the end.  Now, what happens to the blocks of $\sigma'$ when this is done?  A block of length $j$ in $\sigma'$ either becomes a block of length $j+1$ (if the first $n$ is inserted into that block) or length $j$ (if the first $n$ is inserted
somewhere else).  Also, there is a new block at the end of length either 2
(if both copies of $n$ are at the end) or 1 (if not).  So, for $n\geq 2$:
\begin{align*}
b(n,1) &= (2n-2) \cdot b(n-1,1) + (2n-2) \cdot \frac{(2n-2)!}{2^{n-1} (n-1)!},\\
b(n,2) &= b(n-1,1) + (2n-3) \cdot b(n-1,2) + \frac{(2n-2)!}{2^{n-1} (n-1)!},\\
b(n,j) &= (j-1) \cdot b(n-1,j-1) + (2n-1-j) \cdot b(n-1,j) \qquad (j \ge 3).
\end{align*}

Dividing by $(2n)!/(2^n n!)$, we get corresponding recurrences for $\eb$:  $\eb(1,2) = 1$, $\eb(1,j) = 0$ for $j \ne 2$, and for $n \ge 2$,
\begin{align*}
\eb(n,1) &= \frac{(2n-2) \cdot (\eb(n-1,1)+1)}{2n-1},\\
\eb(n,2) &= \frac{\eb(n-1,1) + (2n-3) \cdot \eb(n-1,2) + 1}{2n-1},\\
\eb(n,j) &= \frac{(j-1) \cdot \eb(n-1,j-1) + (2n-1-j) \cdot \eb(n-1,j)}{2n-1} \qquad (j \ge 3).
\end{align*}
Finally, we use the definition of $\db(n,j)$ (equation \eqref{eq:dbldef}) to convert these recurrences into recurrences for $\db$.

\begin{lemma}\label{lem:dbrec}
The numbers $\db(n,j)$ satisfy the following equations:
\begin{align*}
\db(1,2) &= \frac{3}{4},\\
\db(1,j) &= -\frac{6}{j(j+1)(j+2)} \qquad (j \ne 2),\\
\intertext{and for $n \ge 2$,}
\db(n,1) &= \frac{(2n-2) \cdot \db(n-1,1)-1}{2n-1},\\
\db(n,2) &= \frac{\db(n-1,1) + (2n-3) \cdot \db(n-1,2) + 1}{2n-1},\\
\db(n,j) &= \frac{(j-1) \cdot \db(n-1,j-1) + (2n-1-j) \cdot \db(n-1,j)}{2n-1} \qquad (j \ge 3).
\end{align*}
\end{lemma}
\begin{proof}
  The derivation of these equations involves only routine algebra.  We
  illustrate by deriving the formula for $\db(n,2)$:
\begin{align*}
\db(n,2) &= \eb(n,2) - \frac{2(2n+1)}{2 \cdot 3 \cdot 4}\\
&= \frac{\eb(n-1,1) + (2n-3) \cdot \eb(n-1,2) + 1}{2n-1} - \frac{2n+1}{12}\\
&= \frac{1}{2n-1} \left(\frac{2(2n-1)}{1 \cdot 2 \cdot 3} + \db(n-1,1) + (2n-3) \cdot \left(\frac{2(2n-1)}{2 \cdot 3 \cdot 4} + \db(n-1,2)\right) + 1\right) - \frac{2n+1}{12}\\
&= \frac{1}{3} + \frac{\db(n-1,1)}{2n-1} + \frac{2n-3}{12} + \frac{(2n-3) \cdot \db(n-1,2)}{2n-1} + \frac{1}{2n-1} - \frac{2n+1}{12}\\
&= \frac{\db(n-1,1) + (2n-3) \cdot \db(n-1,2) + 1}{2n-1}. \qedhere
\end{align*}
\end{proof}

Using these recurrences, we can now study the numbers $\db(n,j)$.  We
begin by determining these numbers for small values of $j$.

\begin{lemma}\label{lem:dbvalues}
\begin{enumerate}
\item For $n\geq 1$,
\[
\db(n,1) = -1.
\]
\item For $n\geq 1$,
\[
\db(n,2) = \frac{3}{4(2n-1)}.
\]
\item For $n\geq 2$,
\[
\db(n,3) = 2\db(n,2) = \frac{3}{2(2n-1)}.
\]
\end{enumerate}
\end{lemma}
\begin{proof}
  All three statements are proven by induction using
  Lemma~\ref{lem:dbrec}, with the proof of each statement after the
  first also making use of the previous statement.

\end{proof}

\begin{lemma}
For all $n$ and $j \ge 2$,
\begin{equation}\label{eq:dbbound}
  \left|\db(n,j)\right| \le \frac{3}{4} \cdot \frac{2 \cdot 4 \cdot 6 \cdots (2n-2)}{3 \cdot 5 \cdot 7 \cdots (2n-1)}.
\end{equation}
\end{lemma}
\begin{proof}
  By induction on $n$.  For $n=1$, we interpret the products in the
  numerator and denominator of the last fraction in \eqref{eq:dbbound}
  as empty products, which are equal to 1.  So the inequality to be
  proven is $\left|\db(1,j)\right| \le 3/4$, which follows easily from
  the formulas in Lemma~\ref{lem:dbrec}.

  Now suppose that $n \ge 2$, and the lemma holds for $n-1$.  Let
\[
C = \frac{3}{4} \cdot \frac{2 \cdot 4 \cdots (2n-4)}{3 \cdot 5 \cdots (2n-3)}.
\]
Then the inductive hypothesis is that for all $j \ge 2$,
$\left|\db(n-1,j)\right| \le C$.  We now consider three cases:

Case 1:  $j=2$.  Then by Lemma~\ref{lem:dbvalues},
\begin{align*}
\left|\db(n,2)\right| &= \frac{3}{4(2n-1)} \le \frac{3}{4} \cdot \frac{2}{2n-1}\\
&\le \frac{3}{4} \cdot \frac{2}{2n-1} \cdot \frac{4 \cdot 6 \cdots (2n-2)}{3 \cdot 5 \cdots (2n-3)} = \frac{3}{4} \cdot \frac{2 \cdot 4 \cdot 6 \cdots (2n-2)}{3 \cdot 5 \cdot 7 \cdots (2n-1)}.
\end{align*}

Case 2:  $3 \le j \le 2n-1$.  Then
\begin{align*}
\left|\db(n,j)\right| &= \left| \frac{(j-1) \cdot \db(n-1,j-1) + (2n-1-j) \cdot \db(n-1,j)}{2n-1} \right|\\
&\le \frac{j-1}{2n-1} \cdot \left|\db(n-1,j-1)\right| + \frac{2n-1-j}{2n-1} \cdot \left|\db(n-1,j)\right|\\
&\le \frac{j-1}{2n-1} \cdot C + \frac{2n-1-j}{2n-1} \cdot C = \frac{2n-2}{2n-1} \cdot C\\
&= \frac{2n-2}{2n-1} \cdot \frac{3}{4} \cdot \frac{2 \cdot 4 \cdots (2n-4)}{3 \cdot 5 \cdots (2n-3)} = \frac{3}{4} \cdot \frac{2 \cdot 4 \cdots (2n-2)}{3 \cdot 5 \cdots (2n-1)}.
\end{align*}

Case 3: $j \ge 2n$.  Since $n \ge 2$, this implies that $j \ge n+2$.
Therefore $\eb(n,j)=0$, so
\[
\db(n,j) = \eb(n,j) - \frac{2(2n+1)}{j(j+1)(j+2)} = -\frac{2(2n+1)}{j(j+1)(j+2)}.
\]
Thus
\begin{multline*}
\left|\db(n,j)\right| = \frac{2(2n+1)}{j(j+1)(j+2)} \le \frac{2(2n+1)}{2n(2n+1)(2n+2)} = \frac{1}{n(2n+2)} \le \frac{1}{2n-1}\\
\le \frac{3}{4} \cdot \frac{2}{2n-1} \le \frac{3}{4} \cdot \frac{2}{2n-1} \cdot \frac{4 \cdot 6 \cdots (2n-2)}{3 \cdot 5 \cdots (2n-3)} = \frac{3}{4} \cdot \frac{2 \cdot 4 \cdots (2n-2)}{3 \cdot 5 \cdots (2n-1)}. \qedhere
\end{multline*}
\end{proof}

To get a better idea of the size of the bound in the last lemma, note that
\[
\left(\frac{2 \cdot 4 \cdots (2n-2)}{3 \cdot 5 \cdots (2n-1)}\right)^2 = \frac{2}{1} \cdot \frac{2}{3} \cdot \frac{4}{3} \cdot \frac{4}{5} \cdots \frac{2n-2}{2n-3} \cdot \frac{2n-2}{2n-1} \cdot \frac{1}{2n-1}.
\]
Now, it is not hard to see that the product
\[
\frac{2}{1} \cdot \frac{2}{3} \cdot \frac{4}{3} \cdot \frac{4}{5} \cdots \frac{2n-2}{2n-3} \cdot \frac{2n-2}{2n-1}
\]
increases as $n$ increases, and it is well known that as $n \to
\infty$ it converges to $\pi/2$ (see, for example, \cite{wallis}).  Therefore for $j
\ge 2$ we have
\begin{equation}\label{eq:wallis}
\left|\db(n,j)\right| \le \frac{3}{4} \cdot \sqrt{\frac{\pi}{2(2n-1)}}.
\end{equation}
In particular, it follows that $\db(n,j) \to 0$ as $n \to \infty$, which justifies our claim that \eqref{eq:blasymp} gives the asymptotic limits for the $\eb(n,j)$.

To prove Theorem~\ref{thm:main}, for fixed $n$, we need to find the sum of the numbers $\db(n,j)$, for $j$ even.  As a start on this, we compute the sum of all of the $\db(n,j)$:
\begin{equation}\label{eq:dbsumall}
\sum_{j=1}^\infty \db(n,j) = \sum_{j=1}^\infty \eb(n,j) - \sum_{j=1}^\infty \frac{2(2n+1)}{j(j+1)(j+2)}.
\end{equation}
The first sum on the right-hand side of equation \eqref{eq:dbsumall}
is the expected number of blocks (of all lengths).  But for every
standard deal the number of blocks is $n$, so this expected number is
$n$.  We can evaluate the second sum on the right-hand side of
\eqref{eq:dbsumall} by using the partial fractions decomposition
\eqref{eq:bpartfrac}:
\begin{align*}
\sum_{j=1}^\infty \frac{2(2n+1)}{j(j+1)(j+2)} &= (2n+1) \lim_{N \to \infty} \sum_{j=1}^N \left(\frac{1}{j} - \frac{2}{j+1} + \frac{1}{j+2}\right)\\
&= (2n+1) \lim_{N \to \infty} \left(1 - \frac{2}{2} + \frac{1}{3} + \frac{1}{2} - \frac{2}{3} + \frac{1}{4} + \cdots + \frac{1}{N} - \frac{2}{N+1} + \frac{1}{N+2}\right)\\
&= (2n+1) \lim_{N \to \infty} \left(\frac{1}{2} - \frac{1}{N+1} + \frac{1}{N+2}\right) = n + \frac{1}{2}.
\end{align*}
Thus,
\[
\sum_{j=1}^\infty \db(n,j) = \sum_{j=1}^\infty \eb(n,j) - \sum_{j=1}^\infty \frac{2(2n+1)}{j(j+1)(j+2)} = n - \left(n + \frac{1}{2}\right) = -\frac{1}{2}.
\]
According to Lemma~\ref{lem:dbvalues}, $\db(n,1) = -1$, so
\begin{equation}\label{eq:dbsumfrom2}
\sum_{j=2}^\infty \db(n,j) = \frac{1}{2}.
\end{equation}

To separate out the contributions of the even- and odd-numbered terms
in \eqref{eq:dbsumfrom2} we will use the following fact.

\begin{lemma}\label{lem:dbshape}
  For $n \ge 1$, there exist numbers $k_n$ and $\ell_n$ such that $2
  \le k_n < \ell_n \le n+2$ and
\begin{align*}
&\db(n,2) \le \db(n,3) \le \cdots \le \db(n,k_n),\\
&\db(n,k_n) \ge \db(n,k_n+1) \ge \cdots \ge \db(n,\ell_n), \qquad \text{ and}\\
&\db(n,\ell_n) \le \db(n,\ell_n+1) \le \cdots.
\end{align*}
\end{lemma}
\begin{proof}
  We use induction on $n$.  It is easy to verify that the lemma holds
  for $n=1$ and $n=2$ (with $k_1 = 2$, $\ell_1 = 3$, $k_2 = 3$, and
  $\ell_2 = 4$).  Now suppose that $n \ge 3$, and the lemma holds for
  $n-1$.  To prove that the lemma holds for $n$, we will show that
\begin{align}
&\db(n,2) \le \db(n,3) \le \cdots \le \db(n,k_{n-1}),\label{eq:dbshape1}\\
&\db(n,k_{n-1}+1) \ge \db(n,k_{n-1}+2) \ge \cdots \ge \db(n,\ell_{n-1}),\label{eq:dbshape2}\\
&\db(n,\ell_{n-1}+1) \le \db(n,\ell_{n-1}+2) \le \cdots.\label{eq:dbshape3}
\end{align}
We can therefore set $k_n$ to be one of $k_{n-1}$ or $k_{n-1}+1$, and
set $\ell_n$ to be one of $\ell_{n-1}$ or $\ell_{n-1}+1$.

To prove \eqref{eq:dbshape1}, suppose that $2 \le j < k_{n-1}$; we
must verify that $\db(n,j) \le \db(n,j+1)$.  If $j=2$, this follows
immediately from Lemma~\ref{lem:dbvalues}.  Now suppose $j \ge 3$, and
note that $j < \ell_{n-1} \le n+1 \le 2n-2$.  Then by the inductive
hypothesis we have
\[
\db(n-1,j-1) \le \db(n-1,j) \le \db(n-1,j+1),
\]
and therefore
\begin{align}
\db(n,j&) = \frac{(j-1) \cdot \db(n-1,j-1) + (2n-1-j) \cdot \db(n-1,j)}{2n-1}\notag\\
&\le \frac{(2n-2) \cdot \db(n-1,j)}{2n-1}\label{eq:dbshapepf}\\
&\le \frac{j \cdot \db(n-1,j) + (2n-2-j) \cdot \db(n-1,j+1)}{2n-1} = \db(n,j+1).\notag
\end{align}

Similar reasoning can be used to prove \eqref{eq:dbshape2}.  To prove
\eqref{eq:dbshape3}, suppose $j \ge \ell_{n-1}+1$.  If $j \le 2n-2$
then we can repeat the reasoning in \eqref{eq:dbshapepf} to show that
$\db(n,j) \le \db(n,j+1)$.  If $j \ge 2n-1$ then $j \ge n+2$, so
$\eb(n,j) = \eb(n,j+1) = 0$, and therefore
\[
\db(n,j) = -\frac{2(2n+1)}{j(j+1)(j+2)} \le -\frac{2(2n+1)}{(j+1)(j+2)(j+3)} = \db(n,j+1). \qedhere
\]
\end{proof}

Lemma~\ref{lem:dbshape} will be useful to us because of the following fact.

\begin{lemma}\label{lem:evenodddiff}
  Let $a_1,a_2,\ldots,a_k$ be a monotonic sequence with the absolute
  value of each term bounded by some constant $\delta$.  Write $A$ for
  the sum of the odd-indexed terms, i.e., $a_1+a_3+\cdots+a_k$ if $k$
  is odd and $a_1+a_3+\cdots+a_{k-1}$ if $k$ is even.  Write $B =
  \sum_{i=1}^ka_i - A$, so that $B$ is the sum of the even-indexed
  terms.  Then
  \begin{equation}
    |B-A| \leq 2\delta.
  \end{equation}
\end{lemma}
\begin{proof}
  Suppose the sequence is monotonically increasing.  If $k$ is even,
  then
  \[
  A = a_1 + a_3 + \cdots + a_{k-1} \le a_2 + a_4 + \cdots + a_k = B
  \]
  and
  \[
  A - a_1 = a_3 + a_5 + \cdots + a_{k-1} \ge a_2 + a_4 + \cdots + a_{k-2} = B - a_k.
  \]
  This implies that $0 \le B-A \le a_k-a_1$, so $|B-A| \le |a_k|+|a_1|
  \le 2\delta$.  If $k$ is odd, then we obtain the inequality $a_1 \le
  A - B \le a_k$, which leads to the stronger bound of $|B-A|\leq
  \delta$.  The argument is similar if the sequence is monotonically
  decreasing.
\end{proof}

We are finally ready to evaluate the first limit needed to complete
our proof of Theorem~\ref{thm:main}:

\begin{lemma}\label{lem:dbsumeven}
\[
\lim_{n \to \infty} \sum_{i=1}^\infty \db(n,2i) = \frac{1}{4}.
\]
\end{lemma}
\begin{proof}
  By Lemma~\ref{lem:dbshape}, for any positive integer $N$ the
  sequence $\db(n,2), \db(n,3), \ldots, \db(n,2N+1)$ consists of at
  most three monotonic subsequences.  We can therefore apply
  Lemma~\ref{lem:evenodddiff} to each subsequence with the bound
  $\delta = (3/4) \sqrt{\pi/(2(2n-1))}$ from \eqref{eq:wallis} to show
  that
\[
\sum_{i=1}^N \db(n,2i+1) - 6\delta \le \sum_{i=1}^N \db(n,2i) \le \sum_{i=1}^N \db(n,2i+1) + 6\delta.
\]
Adding $\sum_{i=1}^N \db(n,2i)$, we can rewrite this as
\[
\sum_{j=2}^{2N+1} \db(n,j) - 6\delta \le 2\sum_{i=1}^N \db(n,2i) \le \sum_{j=2}^{2N+1} \db(n,j) + 6\delta.
\]
Letting $N \to \infty$ and applying equation \eqref{eq:dbsumfrom2}, we conclude that
\[
\frac{1}{2} - 6\delta \le 2\sum_{i=1}^\infty \db(n,2i) \le \frac{1}{2} + 6\delta,
\]
or in other words
\[
\frac{1}{4} - \frac{9}{4} \sqrt{\frac{\pi}{2(2n-1)}} \le \sum_{i=1}^\infty \db(n,2i) \le \frac{1}{4} + \frac{9}{4} \sqrt{\frac{\pi}{2(2n-1)}}.
\]
The lemma follows now by letting $n \to \infty$.
\end{proof}

\subsection{Sum of $\dl(n,2i)$}\label{sec:sumdleven}

We now do similar calculations for the lucky blocks.  Since the steps are similar to those in the last subsection, we skip many of the details.

If $n=1$, then the only block in the standard deal $1\,1$ is lucky, so $l(1,2)=1$ and $l(1,j)=0$ for $j \ge 3$.  For $n \ge 2$, the numbers $l(n,j)$ satisfy the recurrences
\begin{align*}
l(n,2) &= (2n-3) \cdot l(n-1,2) + \frac{(2n-2)!}{2^{n-1}(n-1)!},\\
l(n,j) &= (j-2) \cdot l(n-1,j-1) + (2n-1-j) \cdot l(n-1,j) \qquad (j \ge 3).
\end{align*}
Therefore $\el(1,2)=1$, $\el(1,j) = 0$ for $j \ge 3$, and for $n \ge 2$,
\begin{align*}
\el(n,2) &= \frac{(2n-3) \cdot \el(n-1,2) + 1}{2n-1},\\
\el(n,j) &= \frac{(j-2) \cdot \el(n-1,j-1) + (2n-1-j) \cdot \el(n-1,j)}{2n-1} 
\qquad (j \ge 3).
\end{align*}
Finally, applying the definition of $\dl(n,j)$ we get 
\begin{align*}
\dl(1,2) &= \frac{1}{2},\\
\dl(1,j) &= -\frac{1}{j(j-1)} \qquad (j>2),\\
\intertext{and for $n \ge 2$,}
\dl(n,2) &= \frac{(2n-3) \cdot \dl(n-1,2)}{2n-1},\\
\dl(n,j) &= \frac{(j-2) \cdot \dl(n-1,j-1) + (2n-1-j) \cdot \dl(n-1,j)}{2n-1} \qquad (j \ge 3).
\end{align*}

\begin{lemma}\label{lem:dlbound}
For all $n$ and $j \ge 2$,
\[
\left|\dl(n,j)\right| \le \frac{1}{2(2n-1)}.
\]
Also, for $n \ge 2$,
\[
\dl(n,2) = \dl(n,3) = \frac{1}{2(2n-1)}.
\]
\end{lemma}
\begin{proof}
By induction on $n$.
\end{proof}

\begin{lemma}\label{lem:lsum}
For all $n$,
\[
\sum_{j=2}^\infty \el(n,j) = 1\qquad \text{ and } \qquad \sum_{j=2}^\infty \dl(n,j) = 0.
\]
\end{lemma}

Note that the first equation in Lemma~\ref{lem:lsum} indicates that the expected number of adjacent matching pairs in a random deal is $1$.

\begin{proof}
  The first equation can be proven by induction on $n$.  For the
  second, we begin with the definition of $\dl(n,j)$:
\[
\sum_{j=2}^\infty \dl(n,j) = \sum_{j=2}^\infty \el(n,j) - \sum_{j=2}^\infty \frac{1}{j(j-1)} = 1 - \sum_{j=2}^\infty \frac{1}{j(j-1)}.
\]
Now we use partial fractions to evaluate the last sum:
\begin{align*}
\sum_{j=2}^\infty \frac{1}{j(j-1)} &= \lim_{N \to \infty} \sum_{j=2}^N \left(\frac{1}{j-1} - \frac{1}{j}\right) \\
 &= \lim_{N \to \infty} \left(1 - \frac{1}{2} + \frac{1}{2} - \frac{1}{3} + \cdots + \frac{1}{N-1} - \frac{1}{N}\right)
= \lim_{N \to \infty} \left(1 - \frac{1}{N}\right) = 1.
\end{align*}
Thus
\[
\sum_{j=2}^\infty \dl(n,j) = 1-1 = 0. \qedhere
\]
\end{proof}

\begin{lemma}\label{lem:dlshape}
For $n\ge 1$, there exists a number $p_n$ such that $2 \le p_n \le n+2$ and
\begin{align*}
&\dl(n,2) \ge \dl(n,3) \ge \cdots \ge \dl(n,p_n) \qquad \text{and}\\
&\dl(n,p_n) \le \dl(n,p_n+1) \le \cdots.
\end{align*}
\end{lemma}
\begin{proof}
By induction on $n$.  It is easy to check that $p_1 = 3$ and $p_2 = 4$.  To establish the lemma for $n \ge 3$, we show that
\begin{align*}
&\dl(n,2) \ge \dl(n,3) \ge \cdots \ge \dl(n,p_{n-1}) \qquad \text{and}\\
&\dl(n,p_{n-1}+1) \le \dl(n,p_{n-1} + 2) \le \cdots.
\end{align*}
The proofs are similar to those in the proof of Lemma~\ref{lem:dbshape}.
\end{proof}

\begin{lemma}\label{lem:dlsumeven}
\[
\lim_{n \to \infty} \sum_{i=1}^\infty \dl(n,2i) = 0.
\]
\end{lemma}
\begin{proof}
By Lemma~\ref{lem:dlshape}, for any positive integer $N$ the sequence $\dl(n,2), \dl(n,3), \ldots, \dl(n,2N+1)$ consists of at most two monotonic subsequences.  As in the proof of Lemma~\ref{lem:dbsumeven}, we apply Lemma~\ref{lem:evenodddiff} with $\delta = 1/(2(2n-1))$ to show that
\[
\sum_{i=1}^{N} \dl(n,2i+1) - 4\delta \le \sum_{i=1}^N \dl(n,2i) \le \sum_{i=1}^{N} \dl(n,2i+1) + 4\delta.
\]
Adding $\sum_{i=1}^N \dl(n,2i)$, letting $N \to \infty$, and dividing by 2, we conclude that
\[
-\frac{1}{2n-1} \le \sum_{i=1}^\infty \dl(n,2i) \le \frac{1}{2n-1},
\]
and the lemma follows.
\end{proof}

This completes the proof of Theorem~\ref{thm:main}.  Observe that the
proofs of Lemma~\ref{lem:dbsumeven} and~\ref{lem:dlsumeven} indicate
that $\epsilon_n$ in Theorem~\ref{thm:main} is
$\displaystyle{O\left(\frac{1}{\sqrt{n}}+\frac{1}{n}\right) =
  O\left(\frac{1}{\sqrt{n}}\right)}$.  As a bonus, our last lemma also
gives us the answer to Question 2:

\begin{cor}
  The expected number of lucky moves in a game tends to $\ln 2$ as $n \to \infty$.
\end{cor}
\begin{proof}
There is a lucky move for each lucky block whose length is even, so
the expected number of lucky moves is $\sum_{i=1}^\infty \el(n,2i)$.
The lemma now follows from equation~\eqref{eq:sumeleven} and
Lemma~\ref{lem:dlsumeven}.
\end{proof}

\section{Acknowledgments}
The authors would like to thank Dan Archdeacon for useful
conversations leading to this paper.


\end{document}